\documentclass[11pt,reqno]{amsart}
\usepackage[margin=1.15in]{geometry}
\usepackage{amsmath,amssymb,amsthm,mathtools}
\usepackage{xcolor}
\usepackage[colorlinks,linkcolor=blue!55!black,citecolor=blue!55!black,
            urlcolor=blue!55!black]{hyperref}
\usepackage{microtype}

\theoremstyle{plain}
\newtheorem{theorem}{Theorem}[section]
\newtheorem{lemma}[theorem]{Lemma}
\newtheorem{proposition}[theorem]{Proposition}

\theoremstyle{definition}
\newtheorem{remark}[theorem]{Remark}

\DeclareMathOperator{\Real}{Re}
\DeclareMathOperator{\dist}{dist}
\newcommand{\D}{\mathbb{D}}
\newcommand{\C}{\mathbb{C}}
\newcommand{\R}{\mathbb{R}}
\newcommand{\Z}{\mathbb{Z}}
\newcommand{\Bcl}{\mathcal{B}}

\DeclarePairedDelimiter{\abs}{\lvert}{\rvert}

\begin{document}

\title{An improved lower bound for Bloch's constant}
\author{Frank Wikström}
\address{Centre for Mathematical Sciences, Lund University,
Box 118, SE-221 00 Lund, Sweden}
\email{frank.wikstrom@math.lth.se}

\begin{abstract}
Let $B$ denote Bloch's constant. We prove
\[
  B \ge \frac{\sqrt{3}}{4}+0.0153
    = 0.448312701\ldots,
\]
improving the lower bounds by Chen--Gauthier ($\sqrt{3}/4+2\cdot10^{-4}$) and
Xiong ($\sqrt{3}/4+3\cdot10^{-4}$). The proof refines Bonk's method through
computer-assisted estimates rigorously verified using interval arithmetic.
\end{abstract}



\maketitle

\section{Introduction}

For $F$ holomorphic on the unit disc $\D$ with $F'(0) = 1$, let $\beta(F)$ be
the supremum of the radii of discs contained in $F(\D)$ that are the
one-to-one image of some subdomain of $\D$ (\emph{schlicht discs}). Bloch's
constant is
\[
  B = \inf \bigl\{ \beta(F) : F \in \mathcal{H}(\D),\ F'(0) = 1 \bigr\}.
\]
Bloch~\cite{Bloch} himself showed that $B > \frac{1}{72}$, but the classical
bounds are
\[
  \frac{\sqrt{3}}{4} = 0.4330127\ldots \le B \le
  \frac{\Gamma(1/3)\,\Gamma(11/12)}{\Gamma(1/4)\sqrt{1+\sqrt{3}}}
  = 0.4718617\ldots,
\]
where the lower bound goes back to Ahlfors~\cite{Ahlfors} and the upper bound to
Ahlfors and Grunsky~\cite{AhlforsGrunsky}, who conjectured it to be the true
value. Heins~\cite{Heins} proved that the lower inequality is strict, but
without obtaining an explicit increment. Minda's systematic
study~\cite{Minda} placed the classical problem in a broader framework of
Bloch constants associated with families of maps and with Euclidean,
hyperbolic, and spherical geometry. The bound resisted quantitative
improvement until 1990, when Bonk~\cite{Bonk} obtained
$B \ge \sqrt{3}/4 + 10^{-14}$. Chen and Gauthier~\cite{ChenGauthier} later
refined Bonk's method to $B \ge \sqrt{3}/4 + 2\cdot10^{-4}$. Xiong subsequently
gave the bound
$B \ge \sqrt{3}/4 + 3\cdot10^{-4}$~\cite{Xiong}.\footnote{We have not been
able to obtain a copy of Xiong's paper. The stated bound is taken from the
paper's introduction as quoted in its MathSciNet entry and agrees with the
bound reported in the zbMath Open review; we have therefore not independently
examined the proof.}

Much of the subsequent work has concerned the surrounding distortion and
covering theory, or variants of the constant, rather than further numerical
improvements for the classical problem. Bonk, Minda, and
Yanagihara~\cite{BonkMindaYanagihara} proved sharp growth, distortion,
curvature, and covering theorems for normalized Bloch functions. Extensions
to several complex variables were studied by Chen and
Gauthier~\cite{ChenGauthierSeveral}, while Chen, Gauthier, and
Hengartner~\cite{ChenGauthierHengartner} obtained Bloch-type results for
planar harmonic mappings.

A different question is whether $B$ can in principle be determined to
arbitrary accuracy. Rettinger~\cite{Rettinger} answered this affirmatively by
proving that Bloch's constant is a computable real number and giving an
algorithm that approximates it to any prescribed precision. This addresses a
different, computability-theoretic question and does not give a new explicit
numerical bound.

The improvements of Bonk and of Chen--Gauthier follow the same idea. After
normalization, it is enough to study $f = F'$ in the class
\begin{equation}\label{eq:class}
  \Bcl = \bigl\{\, f \in \mathcal{H}(\D) : f(0) = 1,\quad
  (1-\abs{z}^2)\,\abs{f(z)} \le 1 \ \text{ for all } z \in \D \,\bigr\}.
\end{equation}
Put $r_0=1/\sqrt{3}$. This radius is natural here: Bonk's sharp distortion
theorem guarantees that $\Real f>0$ on $\abs{z}<r_0$ for every $f\in\Bcl$.
Consequently, $F$ is univalent on this disc.

For $0<r<1$, write
$R_r(\theta)=\int_0^r\Real f(te^{i\theta})\,dt$. Lower bounds for
$R_{r_0}$ control the distance from an interior point of $F(\D_{r_0})$ to its
boundary. The minimum of $R_{r_0}(\theta)$, after an optimal translation of
the target disc, therefore controls the size of a schlicht disc in $F(\D)$.
Bonk's theorem gives
$R_{r_0}(\theta) \ge \sqrt{3}/4$ for every~$\theta$, with equality only for one
extremal function, up to natural rotations. The improvements come from showing
that the deficit $R_{r_0}(\theta) - \sqrt{3}/4$ cannot be small simultaneously
in all directions. Chen and Gauthier control that deficit through an explicit
rational minorant depending on the Taylor coefficients~$a_2$ and~$a_3$ of~$f$.

The present paper replaces the rational minorant by the sharp object. Write
$f(z) = 1 + a_2 z^2 + a_3 z^3 + \cdots$ (the coefficient $a_1$ vanishes
automatically, see Lemma~\ref{lem:elem}), and for $0<r<1$ and
$(b_2,b_3) \in \R^2$ set
\begin{equation}\label{eq:Phi}
  \Phi_r(b_2,b_3) = \inf\Bigl\{ \int_0^r f(x)\,dx :
  f \in \Bcl,\ f \text{ has real coefficients},\ a_2 = b_2,\ a_3 = b_3 \Bigr\}.
\end{equation}
The point $(-1,0)$ will be of particular importance, although it is not the
minimum of $\Phi_{r_0}$ over the coefficient body: the Bonk extremal gives
$\Phi_{r_0}(-1,-8/(3\sqrt{3}))=\sqrt{3}/4$, which is the minimum.  The point
of $(-1,0)$ emerges after allowing an optimal translation of the target disc.
Numerical exploration identifies $a_2=-1$, $a_3=0$ as the worst coefficient
configuration for that problem.  There the resulting coefficient-level lower
bound is antipodally symmetric, so a translation cannot raise its smallest
value.  Thus $\Phi_{r_0}(-1,0)$ is the suspected bottleneck of the translated
fixed-radius argument, but the proof below does not assume this numerical
observation.

The function $\Phi_r$ can be viewed as the value function of an
infinite-dimensional linear program: the objective is linear in the Taylor
coefficients, and the constraint $(1-\abs{z}^2)\,\abs{f(z)} \le 1$ is the family
of linear inequalities $(1-\abs{z}^2)\,\Real\bigl(e^{-i\gamma} f(z)\bigr) \le
1$, $\gamma \in \R$. Every finite subfamily gives a relaxation, and every
non-negative combination of finitely many of those inequalities gives a rigorous
affine minorant of $\Phi_r$. This is the mechanism we exploit, and this
observation is what makes the result verifiable: the search for good multipliers
can be as arbitrary as we like, since the resulting inequality is checked
afterwards and stands on its own.

Our main result is the following.

\begin{theorem}\label{thm:main}
$B \ge \dfrac{\sqrt{3}}{4} + 0.0153$, so that
\[
  B \ge 0.44831270189 .
\]
\end{theorem}

The theorem reduces the gap to the conjectured Ahlfors--Grunsky value from
$0.0385$ to $0.0235$. For comparison, the single-minorant version of the argument gives the
following compact intermediate result.

\begin{proposition}\label{prop:pooled}
  $B \ge \dfrac{\sqrt{3}}{4}+0.0114402996202=0.44445300151\ldots$.
\end{proposition}

The distinction between Proposition~\ref{prop:pooled} and the main theorem
lies partly in the choice of radius. Proposition~\ref{prop:pooled} uses only
the universal radius $r_0$. For the main theorem we split according to the
size of $a_3$. When $\abs{a_3}\le\eta$, point-value certificates extend the
positivity of $\Real f$, and hence the univalence of $F$, to a slightly larger
radius $r_1>r_0$; we may then use the longer integral $R_{r_1}$. When
$\abs{a_3}\ge\eta$, we retain the universal radius $r_0$ and instead use a
finite envelope of affine minorants. Thus the radius is chosen according to
the coefficient regime. The values of $r_1$ and $\eta$ are given in
Section~\ref{sec:variable}.

We also record two ingredients used below.

\begin{theorem}\label{thm:a3}
  If $f \in \Bcl$ then $\abs{a_3} \le 3.2888$.
\end{theorem}

Chen and Gauthier use $\abs{a_3} \le 4.2$; Bonk uses $\abs{a_3} \le 5$ and the
estimate that follows directly from Cauchy's inequality is $\abs{a_3} \le
5.3791$. Numerically, the optimal bound for $\abs{a_3}$ appears to be close to
$3.284$ but this is not used by any of the arguments in the paper.

A second ingredient is the classical moment bound
$\abs{\nu_3}\le 1-\abs{\nu_2}^2$ for a probability measure with $\nu_1=0$.
It prevents the second and third moments from being simultaneously large; see
Section~\ref{sec:moments}.

The paper is organised as follows. Section~\ref{sec:norm} recalls the
normalisation and the passage from radial integrals to schlicht discs.
Section~\ref{sec:cert} sets up the linear program and derives the certificate
inequality underlying the computations. Section~\ref{sec:numbers} states the
fixed-radius certificates used in the compact argument.
Section~\ref{sec:moments} states and proves the moment inequality.
Section~\ref{sec:proof} proves the compact Proposition~\ref{prop:pooled}, and
Section~\ref{sec:variable} gives the coefficient-dependent-radius proof of
Theorem~\ref{thm:main}. Section~\ref{sec:impl} describes the computation and its
verification, and Section~\ref{sec:limits} records the remaining numerical gap.

\subsection*{Certified numerics}
All floating-point computation in this paper is search; none of it is trusted.
Every inequality asserted below is re-derived in Arb ball
arithmetic~\cite{Arb} from the stored certificate data, by the programs
\texttt{certify\_bloch.py} and \texttt{variable\_radius\_certificate.py}
accompanying this paper. The role of floating-point computation is only to
\emph{find} multipliers and midpoint dual solutions; interval arithmetic
allows only their rigorously enclosed consequences into the proof.

\subsection*{Tool disclosure}
OpenAI Codex (GPT-5.6 Sol) as well as Anthropic Claude Code (Opus 5.0), both accessed
August 2026, were used as interactive research assistants for mathematical
brainstorming and proof auditing.

The models assisted with numerical exploration as well as with developing the
interval-arithmetic verification scripts. The author has reviewed all generated
material, independently checked the arguments and computations, and takes full
responsibility for the content of the paper.

\section{Normalisation, and discs in the image}\label{sec:norm}

We use the standard reduction, as observed by Landau~\cite{Landau}: in
computing~$B$ it suffices to consider~$F$ with
\begin{equation}\label{eq:normalisation}
  F(0) = 0, \qquad F'(0) = 1, \qquad
  \sup_{z \in \D} (1-\abs{z}^2)\, \abs{F'(z)} \le 1 .
\end{equation}
Then $f := F'$ lies in the class $\Bcl$
of \eqref{eq:class}, and, since $\abs{f(0)} = 1$, the supremum in
\eqref{eq:normalisation} is attained at the origin.

For convenience of the reader, we briefly recall the reduction: Given a map with
$F'(0)=1$, choose $z_n$ so that $A_n=(1-\abs{z_n}^2)\abs{F'(z_n)}$ approaches
the supremum of the hyperbolic derivative, precompose with a disc automorphism
$\phi_n$ satisfying $\phi_n(0)=z_n$, and divide $F\circ\phi_n-F(z_n)$ by
$F'(z_n)\phi_n'(0)$. The resulting maps have derivative one at the origin, their
hyperbolic-derivative norms tend to one, and their schlicht-disc radii equal
$\beta(F)/A_n$. The usual normal-family limit therefore gives
\eqref{eq:normalisation} without increasing the quantity whose infimum
defines~$B$; see~\cite{Landau} for complete details.

\begin{lemma}\label{lem:elem}
Let $f(z) = \sum_{k \ge 0} a_k z^k \in \Bcl$. Then
\begin{enumerate}
\item[(i)] $a_0 = 1$ and $a_1 = 0$;
\item[(ii)] $\abs{a_2} \le 1$;
\item[(iii)] $\abs{a_k} \le B_k := \bigl(\frac{k+2}{k}\bigr)^{k/2}\,\frac{k+2}{2}$
for every $k \ge 1$, and $B_k \le \frac{e}{2}(k+2)$.
\end{enumerate}
\end{lemma}

\begin{proof}
(i) and (ii) By the definition of $\Bcl$, $a_0=1$. The function
$u(z)=(1-\abs{z}^2)\abs{f(z)}$ satisfies $u\le1=u(0)$. For small $\abs{z}$,
\[
  u(z)=1+\Real(a_1z)+O(\abs{z}^2).
\]
Taking $z=re^{i\varphi}$ and letting $r\downarrow0$ shows that
$\Real(a_1e^{i\varphi})\le0$ for every $\varphi$, and hence $a_1=0$.
Consequently,
\[
  u(z)
  =(1-\abs{z}^2)\abs[\big]{1+a_2z^2+O(\abs{z}^3)}
  =1+\Real(a_2z^2)-\abs{z}^2+O(\abs{z}^3).
\]
Taking again $z=re^{i\varphi}$ gives
$\Real(a_2e^{2i\varphi})\le1+O(r)$ for every $\varphi$, and therefore
$\abs{a_2}\le1$.

(iii) By Cauchy's estimate on $\abs{z} = \rho$ and the constraint,
\[
  \abs{a_k} \le \rho^{-k}\max_{\abs{z}=\rho} \abs{f(z)} \le \frac{1}{\rho^{k}(1-\rho^2)}.
\]
 The right hand side is minimised at $\rho^2 = k/(k+2)$, giving the formula
for~$B_k$. Since $(1+2/k)^{k/2}$ increases to~$e$, it follows that $B_k \le
\frac{e}{2}(k+2)$.
\end{proof}

We will make good use of the following result by Bonk~\cite{Bonk}:

\begin{theorem}\label{thm:bonk}
If $f \in \Bcl$ then $\Real f(z) > 0$ for $\abs{z} < r_0$.
\end{theorem}

We use Theorem~\ref{thm:bonk} only for positivity, and hence univalence, up to
$r_0$; in the near branch, i.e.\ for small values of $a_3$, of
Section~\ref{sec:variable}, certificates extend that positivity to a slightly
larger radius. The quantitative part of Bonk's distortion theorem is not used
directly; all quantitative lower bounds needed below are supplied by the
certificates.

\begin{lemma}\label{lem:disc}
Let $F$ satisfy \eqref{eq:normalisation}, let $0<r<1$ be such that
$\Real f>0$ on $\D_r$, and let $c \in \C$. Put
\[
  R_r(\theta) = \int_0^r \Real f\bigl(te^{i\theta}\bigr)\,dt
  \quad\text{and}\quad
  \varrho_r(c) = \min_{\theta} \Bigl[\, R_r(\theta)
  - \Real\bigl(c e^{-i\theta}\bigr) \Bigr].
\]
If $\abs{c} < \min_\theta R_r(\theta)$, then $F(\D_r)$ contains the open disc
with centre $c$ and radius $\varrho_r(c)$, and $F$ is univalent on $\D_r$. In
particular, $\beta(F) \ge \varrho_r(c)$.
\end{lemma}

\begin{proof}
The hypothesis $\Real f > 0$ makes $F$ univalent on $\D_r$: if
$F(z_2)=F(z_1)$ for distinct $z_1,z_2\in\D_r$, convexity of the disc gives
\[
  0=\frac{F(z_2)-F(z_1)}{z_2-z_1}
   =\int_0^1 f\bigl(z_1+s(z_2-z_1)\bigr)\,ds,
\]
whose real part is positive, a contradiction.
Hence $\Omega := F(\D_r)$ is a
domain containing $0$, and $\partial\Omega \subseteq F(\partial \D_r)$
since $F$ is holomorphic on a neighbourhood of $\overline{\D_r}$ and
thus open.

For $\abs{w} = r$ with $w = re^{i\theta}$, we have
$F(w) = \int_0^r f(te^{i\theta})e^{i\theta}\,dt$, so
\begin{equation}\label{eq:proj}
  \abs{F(w) - c} \ge \Real\Bigl( e^{-i\theta}\bigl(F(w)-c\bigr) \Bigr)
  = R_r(\theta) - \Real\bigl(ce^{-i\theta}\bigr) \ge \varrho_r(c).
\end{equation}
Therefore $\dist(c, \partial\Omega) \ge \dist(c, F(\partial\D_r)) \ge \varrho_r(c)$.
Applying \eqref{eq:proj} with $c = 0$ gives
$\dist(0,\partial\Omega) \ge \min_\theta R_r(\theta) > \abs{c}$, so $c \in \Omega$.
Since $\Omega$ is open and $\dist(c,\partial\Omega) \ge \varrho_r(c)$, the disc
of centre $c$ and radius $\varrho_r(c)$ lies in $\Omega$, and the disc is schlicht
because $F$ is univalent on $\D_r$.
\end{proof}

Finally, we record the symmetrisation that converts a direction $\theta$ into
real coefficient data. For $f \in \Bcl$ and $\theta \in \R$ set
\begin{equation}\label{eq:sym}
  f_\theta(z) = \frac{1}{2}\Bigl( f\bigl(ze^{i\theta}\bigr)
  + \overline{f\bigl(\bar z e^{i\theta}\bigr)} \Bigr).
\end{equation}

\begin{lemma}\label{lem:sym}
$f_\theta \in \Bcl$, its Taylor coefficients are real and equal to
$b_k(\theta) = \Real\bigl(a_k e^{ik\theta}\bigr)$, and
$f_\theta(t) = \Real f(te^{i\theta})$ for real $t \in (-1,1)$. Consequently,
for every $0<r<1$,
\[
  R_r(\theta) = \int_0^r f_\theta(t)\,dt
  \ge \Phi_r\bigl(b_2(\theta), b_3(\theta)\bigr).
\]
\end{lemma}

\begin{proof}
$\Bcl$ is convex, invariant under $f \mapsto f(e^{i\theta}\,\cdot\,)$ and under
$f \mapsto \overline{f(\bar{\,\cdot\,})}$; $f_\theta$ is the average of two such
members, so $f_\theta \in \Bcl$. The coefficient and real-axis statements are
immediate from~\eqref{eq:sym}, and the last display is the
definition~\eqref{eq:Phi} of~$\Phi_r$.
\end{proof}

\section{The linear program and its certificates}\label{sec:cert}

Recall that $r_0 = 1/\sqrt{3}$ and put
\[
  w_k = \int_0^{r_0} x^k \, dx = \frac{r_0^{\,k+1}}{k+1},
  \qquad k \ge 0 .
\]
For $f \in \Bcl$ with real coefficients, $\int_0^{r_0} f = \sum_k a_k w_k$.

For a triple $(\rho,\psi,\gamma) \in [0,1) \times \R \times \R$ define the
\emph{row}
\begin{equation}\label{eq:row}
  r_k(\rho,\psi,\gamma) = (1-\rho^2)\,\rho^k \cos(k\psi - \gamma),
  \qquad k \ge 0 .
\end{equation}

\begin{lemma}\label{lem:rowvalid} For every $f \in \Bcl$ where $f(z) = \sum_k
a_k z^k$ with real coefficients, and every $(\rho,\psi,\gamma)$ with $0 \le \rho
< 1$,
\[
  \sum_{k \ge 0} a_k\, r_k(\rho,\psi,\gamma) \le 1 ,
\]
and the series converges absolutely.
\end{lemma}

\begin{proof}
Put $z = \rho e^{i\psi}$. Since the $a_k$ are real,
$\Real\bigl(e^{-i\gamma}f(z)\bigr) = \sum_k a_k \rho^k \cos(k\psi - \gamma)$, and
by assumption this is at most $\abs{f(z)} \le (1-\rho^2)^{-1}$. Multiply by
$1-\rho^2$. Absolute convergence follows from Lemma~\ref{lem:elem}(iii).
\end{proof}

The next theorem is the basis for the computational parts of the paper. It
converts any finite family of non-negative multipliers into an affine minorant
of $\Phi_{r_0}$, with no optimality, convergence or duality-gap hypothesis
whatsoever.

\begin{theorem}[Certificate inequality]\label{thm:cert}
Let $K \ge 4$ be an integer, let $(\rho_j,\psi_j,\gamma_j)_{j=1}^{n}$ be rows
with $0 \le \rho_j \le \rho_{\max} < 1$, and let $\lambda_1,\dots,\lambda_n \ge 0$.
Define
\[
  d_k = w_k + \sum_{j=1}^n \lambda_j\, r_k(\rho_j,\psi_j,\gamma_j),
  \qquad 0 \le k \le K,
\]
and the tail quantities
\[
  T = \frac{e}{2}\sum_{j=1}^n \lambda_j (1-\rho_j^2)
          \sum_{k > K} (k+2)\rho_j^{\,k},
  \qquad
  T_0 = \frac{e}{2}\,\frac{K+3}{K+2}\cdot\frac{r_0^{\,K+2}}{1-r_0}.
\]
Put
\begin{equation}\label{eq:e0}
  e_0 = d_0 - \sum_{j=1}^n \lambda_j
        - \sum_{k=4}^{K} B_k\,\abs{d_k} - T - T_0 .
\end{equation}
Then for all $(b_2,b_3) \in \R^2$,
\begin{equation}\label{eq:minorant}
  \Phi_{r_0}(b_2,b_3) \ge e_0 + d_2\,b_2 + d_3\,b_3 .
\end{equation}
\end{theorem}

\begin{proof}
Let $f \in \Bcl$ have real coefficients with $a_2 = b_2$, $a_3 = b_3$; if there
is no such $f$ then $\Phi_{r_0}(b_2,b_3) = +\infty$ and there is nothing to
prove.
Write $q_k = \sum_j \lambda_j r_k(\rho_j,\psi_j,\gamma_j)$, so that
$d_k = w_k + q_k$ for $k \le K$, and extend $d_k := w_k + q_k$ to all $k$.
Multiplying the inequality of Lemma~\ref{lem:rowvalid} for the $j$-th row by
$\lambda_j \ge 0$ and summing,
\begin{equation}\label{eq:sumrows}
  \sum_{k \ge 0} a_k q_k \le \sum_{j=1}^n \lambda_j .
\end{equation}
Hence
\[
  \int_0^{r_0} f = \sum_{k\ge0} a_k w_k
  = \sum_{k \ge 0} a_k d_k - \sum_{k\ge0} a_k q_k
  \ge \sum_{k \ge 0} a_k d_k - \sum_{j} \lambda_j .
\]
By Lemma~\ref{lem:elem}, $a_0 = 1$, $a_1 = 0$, $a_2 = b_2$, $a_3 = b_3$ and
$\abs{a_k} \le B_k$, so
\[
  \sum_{k\ge0} a_k d_k \ge d_0 + d_2 b_2 + d_3 b_3
  - \sum_{k = 4}^{K} B_k \abs{d_k} - \sum_{k > K} B_k \abs{d_k}.
\]
It remains to bound the last sum by $T + T_0$. For $k > K$,
$\abs{d_k} \le w_k + \sum_j \lambda_j (1-\rho_j^2)\rho_j^{\,k}$ by
\eqref{eq:row}, and $B_k \le \frac{e}{2}(k+2)$ by Lemma~\ref{lem:elem}(iii).
The contribution of the $\lambda$ part is exactly $T$. For the $w$ part,
$w_k = r_0^{k+1}/(k+1)$ and $(k+2)/(k+1) \le (K+3)/(K+2)$ for $k > K$, so
\[
  \sum_{k>K} B_k w_k \le \frac{e}{2}\,\frac{K+3}{K+2}
  \sum_{k > K} r_0^{\,k+1} = T_0 .
\]
Combining the three displays gives \eqref{eq:minorant}.
\end{proof}

\begin{remark}\label{rem:mix}
We call the following operation \emph{pooling}. The set of admissible
multiplier vectors is a convex cone, and the map
$\lambda\mapsto(d_0,d_2,d_3)$ is affine along convex combinations (because
$\sum_i p_i=1$ reproduces the fixed vector~$w$). Since
$\abs[\big]{\sum_i p_i d^{(i)}_k}\le
\sum_i p_i\abs{d^{(i)}_k}$, the certificate built from
$\sum_i p_i\lambda^{(i)}$ is at least as strong as the corresponding convex
combination of the individual minorants. Thus a finite family of
certificates may be collapsed into a single pooled certificate.
\end{remark}

\begin{remark}\label{rem:free}
If we take $b_2 = b_3 = 0$ in the certificate we actually use, we get
$\Phi_{r_0}(0,0) \ge 0.4962$, and dropping the coefficient constraints
altogether the same machinery reproduces Bonk's sharp radial bound
$\sqrt{3}/4 = 0.4330127$ to five decimals. This is the sense in which the
linear program is the sharp form of the rational minorants
of~\cite{Bonk,ChenGauthier}.
\end{remark}

A parallel argument bounds individual coefficients.

\begin{theorem}\label{thm:a3cert}
With notation as in Theorem~\ref{thm:cert}, set
$q_k = \sum_j \lambda_j r_k(\rho_j,\psi_j,\gamma_j)$ and suppose $q_3 > 0$.
Then every $f \in \Bcl$ with real coefficients satisfies
\[
  a_3 \le \frac{1}{q_3}\Bigl(\, \sum_j \lambda_j - q_0 + \abs{q_2}
  + \sum_{k=4}^{K} B_k \abs{q_k} + T \Bigr).
\]
The same bound holds for $\abs{a_3}$ for any $f \in \Bcl$ whether or
not its coefficients are real.
\end{theorem}

\begin{proof}
From \eqref{eq:sumrows}, $q_3 a_3 \le \sum_j \lambda_j - a_0q_0 - a_1q_1 -
a_2 q_2 - \sum_{k \ge 4} a_k q_k$. Use $a_0 = 1$, $a_1 = 0$, $\abs{a_2} \le 1$,
$\abs{a_k} \le B_k$, and bound the tail $k > K$ by $T$ as in the previous proof.
For the passage to $\abs{a_3}$: the class $\Bcl$ and the reality of the
coefficients are preserved by $f(z) \mapsto f(-z)$, which sends $a_3$ to
$-a_3$. Finally, for arbitrary $f \in \Bcl$ choose $\tau$ with
$e^{3i\tau}a_3 = \abs{a_3}$ and apply the real-coefficient case to the
symmetrisation~\eqref{eq:sym} of $f(e^{i\tau}\,\cdot\,)$, whose third
coefficient is $\Real(e^{3i\tau}a_3) = \abs{a_3}$.
\end{proof}

\section{The verified certificates}\label{sec:numbers}

We begin with two multiplier vectors and verify the resulting inequalities
in Arb ball arithmetic at $40$ digits. The data, and the verification program,
accompany this paper; see Section~\ref{sec:impl}.

\begin{proposition}\label{prop:certs}
There are explicit finitely supported non-negative multiplier vectors, with
$K = 320$ and all $\rho_j \le 0.95$, for which Theorems~\ref{thm:cert} and
\ref{thm:a3cert} yield
\begin{align}
  \abs{a_3} &\le 3.28877762819 \qquad\text{for all } f \in \Bcl,
  \label{eq:a3num}\\[2pt]
  \Phi_{r_0}(b_2,b_3) &\ge 0.4961640 + 0.0509380\,b_2 + 0.0092164\,b_3
  \qquad \text{for } \abs{b_2} \le 1,\ \abs{b_3} \le C,
  \label{eq:minnum}
\end{align}
where $C = 3.2888$.
\end{proposition}

The affine minorant in \eqref{eq:minnum} is a slightly weakened decimal version
of the exact one encoded by the stored pooled multiplier vector. Since $b_2$
and $b_3$ may have either sign, we verified the displayed inequality directly
on the stated box; it holds there with a margin of $5.7\cdot 10^{-8}$. The proof
of Proposition~\ref{prop:pooled} uses the exact certificate coefficients
$e_0,d_2,d_3$ from Theorem~\ref{thm:cert}, evaluated with Arb enclosures.

Theorem~\ref{thm:a3} is \eqref{eq:a3num} rounded up. Specialising
\eqref{eq:minnum} at $b_2 = -1$, $b_3 = 0$ gives
$\Phi_{r_0}(-1,0) \ge 0.4452260$, i.e.\ a gain of $0.0122133$ over
$\sqrt{3}/4$;
a certificate computed specifically at that point does better:

\begin{proposition}\label{prop:phi10}
$\Phi_{r_0}(-1,0) \ge \dfrac{\sqrt{3}}{4} + 0.0150346379669 .$
\end{proposition}

For comparison, \cite[Lemma 6]{ChenGauthier} gives
$\Phi_{r_0}(-1,0) \ge \sqrt{3}/4 + 0.00548$;
Proposition~\ref{prop:phi10} improves their constant by a factor of more than
$2.74$. We record this point-tuned estimate for comparison with their result;
it is not used in the proof of the main theorem.

Proposition~\ref{prop:phi10} comes from a separate, coarser discretisation
($44 \times 129$ rows, $\rho_{\max} = 0.94$, $K = 260$) which happens to give a
better verified value at this particular point than the finer run used for
Proposition~\ref{prop:certs}; the finer run certifies
$\Phi_{r_0}(-1,0) \ge \sqrt{3}/4 + 0.0149407644$. Each is an independently verified
inequality, and we quote whichever is stronger for each statement; nothing
depends on the two being combined.

\begin{remark}
For orientation, consider the minorant \eqref{eq:minnum} at two familiar
members of $\Bcl$. At the Bonk extremal
\[
f_1(z)=\frac{1-\sqrt{3} z}{(1-z/\sqrt{3})^3},
\]
for which $a_2=-1$, $a_3=-8/(3\sqrt{3})$ and
$\int_0^{r_0}f_1=\sqrt{3}/4$, its right-hand side is $0.4310364$, lying
$0.0019763$ below the exact value $0.4330127$. At the constant function
$f\equiv1$, the minorant gives $0.4961640$, compared with the exact value
$r_0=0.5773503$.
\end{remark}

\section{A moment inequality}\label{sec:moments}

\begin{theorem}\label{thm:moment}
Let $\mu$ be a probability measure on $\R/2\pi\Z$ and
$\nu_k = \int e^{ik\theta}\,d\mu(\theta)$. If $\nu_1 = 0$, then
\[
  \abs{\nu_3} \le 1 - \abs{\nu_2}^2.
\]
\end{theorem}

The inequality is classical and follows from the Schur parametrisation of the
Carath\'eodory class; see \cite[Section~1.3]{SimonOPUC} and~\cite{LiSugawa}.
We include a short proof for completeness.

\begin{proof}[Proof of Theorem~\ref{thm:moment}]
Associate to~$\mu$ its Carath\'eodory function
\[
 P(z)=\int\frac{e^{i\theta}+z}{e^{i\theta}-z}\,d\mu(\theta)
     =1+2\sum_{k\ge1}\overline{\nu_k}z^k
\]
and put $\omega=(P-1)/(P+1)$.  Then $\omega$ maps $\D$ to $\D$ and
$\omega(0)=0$.  If $\nu_1=0$, its expansion begins
\[
 \omega(z)=\overline{\nu_2}z^2+\overline{\nu_3}z^3+O(z^4).
\]
The Schwarz lemma applied twice shows that $g(z)=\omega(z)/z^2$ satisfies
$\abs{g}\le1$ on~$\D$.  If $g$ is a unimodular constant the desired inequality
is immediate; otherwise the maximum principle gives $g(\D)\subset\D$, and
Schwarz--Pick at zero gives
\[
 \abs{\nu_3}=\abs{g'(0)}\le1-\abs{g(0)}^2
             =1-\abs{\nu_2}^2.
\]
\end{proof}

The inequality is sharp.  For example, let $\pi/2\le\alpha\le\pi$ and put
\[
 q=\frac{1}{2(1-\cos\alpha)},\qquad
 p=\frac{-\cos\alpha}{1-\cos\alpha},\qquad
 \mu=p\delta_0+q\delta_\alpha+q\delta_{-\alpha}.
\]
Then $p+2q=1$, $\nu_1=0$, and a direct calculation gives
\[
 \nu_2=-(1+2\cos\alpha),\qquad
 \nu_3=-4\cos\alpha(1+\cos\alpha)=1-\abs{\nu_2}^2.
\]
More generally, equality in the non-degenerate case is equivalent to equality
in Schwarz--Pick above, hence to $g$ being a disc automorphism.

\section{The compact pooled bound}\label{sec:proof}

\begin{proof}[Proof of Proposition~\ref{prop:pooled}]
Let $F$ satisfy \eqref{eq:normalisation} and $f = F' \in \Bcl$ with
coefficients $a_2, a_3$. Let
$\ell(b_2,b_3) = e_0 + d_2b_2 + d_3b_3$ be the exact affine minorant defined
by the stored pooled certificate above. Theorem~\ref{thm:cert}, its Arb
verification, and Lemma~\ref{lem:sym} give
\begin{equation}\label{eq:Rlow}
  R_{r_0}(\theta) \ge \Phi_{r_0}\bigl(b_2(\theta),b_3(\theta)\bigr)
  \ge \ell\bigl(b_2(\theta), b_3(\theta)\bigr),
  \qquad b_k(\theta) = \Real\bigl(a_ke^{ik\theta}\bigr).
\end{equation}

\emph{Step 1: duality in the shift.} Let $\mathcal{P}$ be the set of Borel
probability measures on $\R/2\pi\Z$, a convex weak-$*$ compact set, and for
$\mu \in \mathcal{P}$ write $\nu_k(\mu) = \int e^{ik\theta}d\mu$. The function
$R_{r_0}$ is continuous, and $\min_\theta h(\theta) = \min_{\mu \in \mathcal{P}}
\int h\,d\mu$ for continuous $h$. Hence
\[
  \sup_{c \in \C} \min_{\theta}\Bigl[R_{r_0}(\theta)
    - \Real(ce^{-i\theta})\Bigr]
  = \sup_{c \in \C}\ \min_{\mu \in \mathcal{P}}
    \Bigl[ \int R_{r_0}\,d\mu
      - \Real\bigl(c\,\overline{\nu_1(\mu)}\bigr) \Bigr].
\]
The bracket is affine in $c \in \C \cong \R^2$ and affine and weak-$*$
continuous in $\mu$, so Sion's minimax theorem~\cite{Sion} applies and the
right-hand side equals
\[
  \min_{\mu \in \mathcal{P}}\ \sup_{c \in \C}
  \Bigl[ \int R_{r_0}\,d\mu
    - \Real\bigl(c\,\overline{\nu_1(\mu)}\bigr) \Bigr]
  = \min\Bigl\{ \int R_{r_0}\,d\mu : \mu \in \mathcal{P},\
    \nu_1(\mu) = 0 \Bigr\},
\]
since the supremum over $c$ is $+\infty$ unless $\nu_1(\mu) = 0$.

\emph{Step 2: the moment region.} Let $\mu \in \mathcal{P}$ with
$\nu_1(\mu) = 0$, and set
\[
  m_2 = \int b_2(\theta)\,d\mu = \Real\bigl(a_2\nu_2\bigr), \qquad
  m_3 = \int b_3(\theta)\,d\mu = \Real\bigl(a_3\nu_3\bigr).
\]
By \eqref{eq:Rlow} and the affineness of $\ell$,
\begin{equation}\label{eq:jensen}
  \int R_{r_0}\,d\mu
  \ge \int \ell\bigl(b_2(\theta),b_3(\theta)\bigr)d\mu
  = e_0 + d_2 m_2 + d_3 m_3 .
\end{equation}
By Lemma~\ref{lem:elem}(ii), $\abs{a_2} \le 1$, so $\abs{m_2} \le \abs{\nu_2}$; by
Theorem~\ref{thm:a3}, $\abs{a_3} \le C := 3.2888$, so using
Theorem~\ref{thm:moment},
\[
  \abs{m_3} \le C\abs{\nu_3} \le C\bigl(1 - \abs{\nu_2}^2\bigr)
  \le C\bigl(1 - m_2^2\bigr),
\]
the last step because $\abs{m_2} \le \abs{\nu_2}$. Hence $(m_2,m_3)$ lies in
\[
  \mathcal{R} = \bigl\{ (m_2,m_3) \in \R^2 :
  \abs{m_2} \le 1,\ \abs{m_3} \le C(1-m_2^2) \bigr\}.
\]

\emph{Step 3: the region minimum in closed form.} For fixed $m_2$, the minimum
of $e_0 + d_2m_2 + d_3m_3$ over the admissible $m_3$ is attained at
$m_3 = -\operatorname{sign}(d_3)\,C(1-m_2^2)$, giving
\[
  g(m_2) = e_0 + d_2 m_2 + \abs{d_3}C\,m_2^2 - \abs{d_3}C ,
\]
a convex quadratic in $m_2$. Therefore
\begin{equation}\label{eq:regionmin}
  \min_{\mathcal{R}} \bigl(e_0 + d_2m_2 + d_3m_3\bigr)
  = \min_{\abs{m_2} \le 1} g(m_2)
  = g\bigl(\operatorname{clamp}_{[-1,1]}(-d_2/(2\abs{d_3}C))\bigr) .
\end{equation}
Evaluating \eqref{eq:regionmin} from the Arb enclosures of the exact
coefficients $e_0,d_2,d_3$ and $C$ gives
\[
  \min_{\mathcal{R}} \bigl(e_0 + d_2m_2 + d_3m_3\bigr)
  \ge \frac{\sqrt{3}}{4} + L, \qquad L = 0.0114402996202,
\]
the interior vertex being the active branch (it sits at
$m_2 = -0.8402\ldots$).

\emph{Step 4: conclusion.} Combining Steps 1--3 with \eqref{eq:jensen},
\[
  \sup_{c} \min_\theta\Bigl[R_{r_0}(\theta)
    - \Real(ce^{-i\theta})\Bigr]
  \ge \frac{\sqrt{3}}{4} + L .
\]
Fix $\delta \in (0, L)$ and choose $c$ attaining this supremum up to $\delta$.
Taking $\theta = \arg c$ in the definition of $\varrho_{r_0}$ gives
$\varrho_{r_0}(c) \le R_{r_0}(\arg c) - \abs{c}$, whence, using
$R_{r_0}(\theta) \le \int_0^{r_0}(1-t^2)^{-1}dt = \operatorname{artanh}(r_0)
= 0.6584790$,
\[
  \abs{c} \le R_{r_0}(\arg c) - \varrho_{r_0}(c)
  \le \operatorname{artanh}(r_0) - \Bigl(\tfrac{\sqrt{3}}{4} + L - \delta\Bigr)
  < 0.2141 + \delta .
\]
On the other hand \eqref{eq:Rlow} and $\abs{b_2(\theta)} \le 1$,
$\abs{b_3(\theta)} \le C$ give the crude uniform bound
\[
  \min_\theta R_{r_0}(\theta)
  \ge e_0 - \abs{d_2} - \abs{d_3}C \ge 0.4149 ,
\]
so $\abs{c} < \min_\theta R_{r_0}(\theta)$ for $\delta$ small, and
Theorem~\ref{thm:bonk} and Lemma~\ref{lem:disc} (with $r=r_0$) yield
$\beta(F) \ge \sqrt{3}/4 + L - \delta$. Letting $\delta \downarrow 0$ and taking
the infimum over $F$ completes the proof.
\end{proof}

\section{The coefficient-dependent two-radius improvement}\label{sec:variable}

Increasing $r$ potentially improves the radial lower bound, but
Lemma~\ref{lem:disc} can be applied only when $\Real f>0$ throughout $\D_r$.
The point-value certificates below justify this increase beyond~$r_0$ for one
coefficient regime; an integral certificate then exploits the resulting extra
annulus.

Theorem~\ref{thm:cert} and its proof remain valid for $\Phi_r$ after replacing
$r_0$ by $r$ in the objective weights~$w_k$ and in~$T_0$. The same argument,
with objective weights $w_k=r^k$, gives affine lower bounds for the point
value $f(r)$. We use both versions below.

Put
\begin{align*}
 r_1&=0.581522\ldots,\\
 \eta&=0.700000\ldots,\qquad C=3.2888.
\end{align*}
The displayed decimals are for orientation.  Throughout this section the
symbols $r_1$ and $\eta$ denote the same binary64 values stored with the
certificate data: the radius in (i) and (ii), the coefficient boxes in (i) and
(ii), the case split, and the sector relaxation in (iii) all use those common
values.  The verifier performs all subsequent arithmetic with Arb enclosures.
The following is the computer-assisted part of the improved proof.

The certificate has two parts. For $\abs{a_3} \le \eta$, point-value
certificates extend univalence from $r_0$ to $r_1$, and one affine integral
certificate bounds the longer integral $R_{r_1}$. For
$\abs{a_3} \ge \eta$, we keep the universal radius $r_0$ and verify a finite
envelope of minorants of $\Phi_{r_0}$ over the coefficient body and the
balanced-measure parameters.

\begin{proposition}[Variable-radius certificate]\label{prop:variablecert}
The stored non-negative multiplier vectors have the following rigorously
verified consequences.
\begin{enumerate}
\item[(i)] If $\abs{b_2}\le1$ and $\abs{b_3}\le\eta$, then
\[
 f(x)\ge0.000936\qquad(r_0\le x\le r_1)
\]
for every real-coefficient $f\in\Bcl$ with those coefficients.
\item[(ii)] On the same coefficient box,
\begin{equation}\label{eq:nearcut}
 \Phi_{r_1}(b_2,b_3)
 \ge \alpha_0+\alpha_2b_2+\alpha_3b_3,
\end{equation}
where $\alpha_0,\alpha_2,\alpha_3$ are the exact coefficients defined by the
stored certificate. For orientation,
\[
 (\alpha_0,\alpha_2,\alpha_3)
 \approx(0.505011,0.056695,0.016797).
\]
Arb evaluation of the exact certificate coefficients gives
\begin{align}
 \alpha_2-2\eta\abs{\alpha_3}
   &\ge0.0331,\label{eq:slopemargin}\\
 \alpha_0-\alpha_2-\frac{\sqrt{3}}{4}
   &\ge0.015304.\label{eq:neargain}
\end{align}
Moreover,
\[
 2(\alpha_2+\eta\abs{\alpha_3})\le0.137,
 \qquad
 \alpha_0-\alpha_2-\eta\abs{\alpha_3}\ge0.436.
\]
\item[(iii)] Let $\ell_j(b_2,b_3)=e_j+p_jb_2+q_jb_3$ be the $47$ stored,
outward-rounded affine minorants of $\Phi_{r_0}$ and put
\[
 H_a(\theta)=\max_{1\le j\le 47}
 \ell_j\bigl(\Real(a_2e^{2i\theta}),\Real(a_3e^{3i\theta})\bigr).
\]
For every coefficient pair of a function in $\Bcl$ with $\abs{a_3}\ge\eta$,
and every probability measure $\mu$ on the circle with $\nu_1=0$,
\begin{equation}\label{eq:awaycert}
 \int H_a(\theta)\,d\mu(\theta)
 \ge\frac{\sqrt{3}}{4}+0.0153.
\end{equation}
\end{enumerate}
\end{proposition}

\begin{proof}
Theorem~\ref{thm:cert} applies without change to the two objectives used here.
For an integral up to $r_1$ its objective coefficients are
$w_k=r_1^{k+1}/(k+1)$, and for a point value at $x$ they are $w_k=x^k$.
In the latter case the objective tail is bounded, uniformly on a radial
interval with upper endpoint $r$, by
\begin{equation}\label{eq:pointtail}
 \frac{e}{2}\, r^{K+1}
  \left(\frac{K+3}{1-r} +\frac{r}{(1-r)^2}\right).
\end{equation}
Indeed, this is the closed form of
$\frac{e}{2} \sum_{k>K}(k+2)r^k$, and hence follows from
Lemma~\ref{lem:elem}(iii).

For (i), the stored data consist of one point-value multiplier vector at each
of the $17$ equally spaced nodes from $r_0$ to $r_1$.  The $16$ intervening
radial intervals are each divided into $32$ closed subintervals.  The lower
endpoint of the first interval is rounded down by one binary64 step, so it is
strictly below the exact value $r_0=1/\sqrt{3}$; thus the certified intervals
cover the claimed closed annulus without an endpoint gap.  Fix one such
subinterval $I$ and let $r$ be its upper endpoint.  In particular, these
intervals cover the circle $\abs{z}=r_0$, while Bonk's theorem supplies strict
positivity inside that circle.  Each multiplier vector at an endpoint of the
parent interval gives, for every $x\in I$,
an inequality
\[
   f(x) \ge e_j(x) + p_j(x)b_2 + q_j(x)b_3
        \ge e_j(x) - \abs{p_j(x)} - \eta\abs{q_j(x)}.
\]
Arb evaluates the three functions on the whole interval~$I$, using
\eqref{eq:pointtail} with this upper endpoint for the omitted tail.  Taking
the larger of the two resulting lower bounds is legitimate because both
minorants hold throughout~$I$.  The least lower endpoint over all
$16\cdot32=512$ subintervals is
$0.000936855231321\ldots$, proving (i) uniformly on the full closed annulus.

For (ii), the stored integral multiplier vector and Theorem~\ref{thm:cert},
with the $r_1$ integral weights, give
\eqref{eq:nearcut}.  All finite sums and both tails are recomputed in Arb.
Direct outward-rounded evaluation of the resulting Arb coefficients gives
\eqref{eq:slopemargin}, \eqref{eq:neargain}, and the two final inequalities
in (ii).

It remains to prove (iii).  We first specify the coefficient relaxation.  Put
\[
 x=(\Real a_2,\Im a_2,\Real a_3,\Im a_3)\in\R^4.
\]
We enlarge the feasible set of~$x$ to a polyhedron consisting of $96$
outward-rounded supporting directions for each of $\abs{a_2}\le1$ and
$\abs{a_3}\le C$, together with $72$ directions at each of the $12$ auxiliary
sampling radii
\[
 0.004,0.008,0.015,0.025,0.04,0.06,0.09,0.13,0.18,0.24,0.31,0.39.
\]
The latter inequalities are
\begin{equation}\label{eq:coeffbody}
  b_2(\tau)+\rho\,b_3(\tau)
     \le \frac{1}{1-\rho^2}+\sum_{k\ge4}B_k\,\rho^{k-2},
\end{equation}
which follow directly from the defining Bloch inequality and
Lemma~\ref{lem:elem}(iii). The right sides, trigonometric coefficients and
tails are enclosed in Arb before being rounded outwards.

Each of the $47$ affine minorants is independently reconstructed by
Theorem~\ref{thm:cert}.  If $(\widetilde e_j,\widetilde p_j,
\widetilde q_j)$ are its Arb-enclosed coefficients, the stored binary slopes
$p_j,q_j$ are made safe on $\abs{b_2}\le1$, $\abs{b_3}\le C$ by replacing the
constant with a downward rounding of
\[
   \widetilde e_j-\abs{\widetilde p_j-p_j} - C\,\abs{\widetilde q_j-q_j}.
\]
This proves that every displayed $\ell_j$ is a global lower bound on the box
where it is used.

For a fixed coefficient pair, the map
$\mu\mapsto\int H_a\,d\mu$ is continuous and affine, so its minimum over the
compact convex set of probability measures with $\nu_1=0$ is attained at an
extreme measure. Such an extreme measure has at most three atoms. Indeed, if
its support meets four pairwise disjoint Borel sets $E_1,\ldots,E_4$ of
positive measure, the four vectors
\[
 \int_{E_j}(1,\cos\theta,\sin\theta)\,d\mu(\theta)\in\R^3
 \qquad(1\le j\le4)
\]
are linearly dependent. The resulting nonzero signed measure, obtained by
taking the same linear combination of the restrictions $\mu|_{E_j}$, has
zero total mass and zero first moment. A sufficiently small positive or
negative multiple may therefore be added to $\mu$, expressing it as the
midpoint of two distinct admissible probability measures. This contradicts
extremality.

Rotation can also be normalised away. If the measure is shifted by
$\theta\mapsto\theta-\tau$ and simultaneously
\[
 (a_2,a_3)\mapsto(e^{2i\tau}a_2,e^{3i\tau}a_3),
\]
then $\nu_1=0$, membership of the coefficient pair in the true coefficient
body, and the value of $\int H_a\,d\mu$ are unchanged. The transformed pair is
still covered by one of the $24$ phase sectors below. We may consequently put
one atom at angle zero. Relabel the atoms so that the next angular gap is the
smallest; it is at most $2\pi/3$. Writing this gap as $\pi u$, the condition
that the origin lie in the convex hull of the three points places the third
angle between $\pi$ and $\pi(1+u)$. Thus, for some
$0\le u\le2/3$ and $0\le v\le1$, the extreme measures are parametrised by
\begin{equation}\label{eq:uvparam}
 \begin{gathered}
 \theta=(0,\pi u,\pi(1+uv)),\\
 (p_0,p_1,p_2)=\frac{1}{A+B+D}(A,B,D),\\
 A=(1-v)\operatorname{sinc}_\pi(u(1-v)),\quad
 B=v\operatorname{sinc}_\pi(uv),\quad
 D=\operatorname{sinc}_\pi(u),
 \end{gathered}
\end{equation}
where $\operatorname{sinc}_\pi(t)=\sin(\pi t)/(\pi t)$, continuously extended
at zero. The boundary includes the balanced two-atom measures. Finally,
$\abs{a_3}\ge\eta$ is covered by $24$ closed phase sectors.  In sector $s$ we
impose the weaker linear condition
\[
 \Real(e^{-2\pi i s/24}a_3)\ge\eta\cos(\pi/24).
\]
This is valid because some sector centre differs from $\arg a_3$ by at most
$\pi/24$.  Let the resulting polyhedron in sector~$s$ be
$P_s=\{x:Gx\le h\}$.  It contains every true coefficient vector belonging to
that sector.

We now give the box inequality used in the subdivision.  If the three atoms
in \eqref{eq:uvparam} have weights $\omega_i$ and angles $\theta_i$, define
\[
 v_j(\theta)=
 \bigl(p_j\cos2\theta,-p_j\sin2\theta,
       q_j\cos3\theta,-q_j\sin3\theta\bigr).
\]
For a closed box $Q\subset[0,2/3]\times[0,1]$, a floating-point LP at the
midpoint supplies numbers $\alpha_{ij}\ge0$, $\sum_j\alpha_{ij}=1$, and
$\zeta_m\ge0$.  They are used only as candidate multipliers: their binary
values are thereafter interpreted exactly, and each $\alpha$-row is
renormalised as an exact non-negative rational probability vector.  On the
whole box Arb encloses
\begin{equation}\label{eq:boxEV}
 E=\sum_{i=0}^2\omega_i\sum_j\alpha_{ij}e_j,\qquad
 V=\sum_{i=0}^2\omega_i\sum_j\alpha_{ij}v_j(\theta_i),\qquad
 Z=V+G^T\zeta .
\end{equation}
For every $(u,v)\in Q$ and every true coefficient vector $x$ belonging to
sector~$s$, we have $x\in P_s$, $\abs{x_1},\abs{x_2}\le1$, and
$\abs{x_3},\abs{x_4}\le C$.  Convexity of the maximum and $\zeta\ge0$ therefore
give
\begin{align}
 \smash[b]{\sum_{i=0}^2\omega_iH_a(\theta_i)}
 &\ge E+V \cdot x \notag\\
 &=E-\zeta \cdot h+Z \cdot x
   +\zeta \cdot (h-Gx) \notag\\
 &\ge E-\zeta \cdot h
   -\abs{Z_1}-\abs{Z_2}-C(\abs{Z_3}+\abs{Z_4}).
 \label{eq:boxlower}
\end{align}
In evaluating the last line, the program takes the lower endpoint of the Arb
enclosure of $E$ and the upper endpoints of the enclosures of each
$\abs{Z_i}$.  Thus \eqref{eq:boxlower} is a rigorous lower bound even when the
midpoint LP basis ceases to be optimal elsewhere in~$Q$.

Starting with a $40\times40$ grid, the program evaluates~\eqref{eq:boxlower} in
Arb on each closed box.  A box is discarded only when the outward-rounded lower
endpoint is strictly larger than $\sqrt{3}/4+0.0153$; otherwise it is bisected and
the test is repeated.  For each of the $24$ sectors this process terminated with
an empty frontier. The parametrisation and the interval evaluation are
continuous on the closed rectangle (the sinc function is evaluated by its
continuous extension), so the boundary, including every balanced two-atom
measure, is part of the same finite cover.  Consequently~\eqref{eq:awaycert}
holds for every extreme balanced measure.  Affinity and the preceding
extreme-point reduction extend it to every probability measure with $\nu_1=0$,
which proves (iii).
\end{proof}

\begin{proof}[Proof of Theorem~\ref{thm:main}]
Let $F$ satisfy \eqref{eq:normalisation}, put $f=F'$, and split into two
cases. First suppose $\abs{a_3}\le\eta$. Bonk's theorem and
Proposition~\ref{prop:variablecert}(i), applied after the symmetrisation
\eqref{eq:sym} in every direction, give $\Real f(z)>0$ for $\abs{z}<r_1$.
Consequently $F$ is univalent on $\D_{r_1}$.

For a probability measure with $\nu_1=0$, integrate \eqref{eq:nearcut}, use
$\abs{a_2}\le1$, $\abs{a_3}\le\eta$, and set $s=\abs{\nu_2}$. Theorem~\ref{thm:moment}
gives
\[
 \int R_{r_1}\,d\mu
 \ge \alpha_0-\alpha_2s-\eta\abs{\alpha_3}(1-s^2).
\]
By \eqref{eq:slopemargin} the right side decreases for $0\le s\le1$; its
minimum is $\alpha_0-\alpha_2$, which exceeds $\sqrt{3}/4+0.0153$ by
\eqref{eq:neargain}. Define the affine minorant
\[
 \widetilde R(\theta)=\alpha_0+
   \alpha_2\Real(a_2e^{2i\theta})+
   \alpha_3\Real(a_3e^{3i\theta})
\]
and put $A=\alpha_2+\eta\abs{\alpha_3}$. Then
$R_{r_1}\ge\widetilde R$ and
$\abs{\widetilde R(\theta)-\alpha_0}\le A$. At $c=0$ the shifted minimum of
$\widetilde R$ is at least $\alpha_0-A$, whereas if $\abs{c}>2A$, taking
$\theta=\arg c$ makes it less than
$\alpha_0+A-\abs{c}<\alpha_0-A$. Thus its supremum over centres is the maximum
over the compact disc $\abs{c}\le2A$. Sion duality, applied to
$\widetilde R$ exactly as in Step~1 of Section~\ref{sec:proof}, and the moment
bound above now supply a centre $c$ in that disc for which
\[
 \min_\theta\bigl[\widetilde R(\theta)-
       \Real(ce^{-i\theta})\bigr]
 \ge \frac{\sqrt{3}}{4}+0.0153.
\]
The same centre works for $R_{r_1}$ because $R_{r_1}\ge\widetilde R$. The
last two bounds in Proposition~\ref{prop:variablecert}(ii) therefore show
\[
 \abs{c}\le0.137<0.436\le\min_\theta R_{r_1}(\theta).
\]
Lemma~\ref{lem:disc} applies and proves more than the asserted result in this
case.

Now suppose $\abs{a_3}\ge\eta$. Each $\ell_j$ is a lower bound for
$\Phi_{r_0}$, so Lemma~\ref{lem:sym} gives $R_{r_0}(\theta)\ge H_a(\theta)$.
The envelope $H_a$ is continuous and bounded. Hence
\[
 c\longmapsto\min_\theta
   \bigl[H_a(\theta)-\Real(ce^{-i\theta})\bigr]
\]
is continuous and tends to $-\infty$ as $\abs{c}\to\infty$ (take
$\theta=\arg c$), so its supremum is attained. Proposition~
\ref{prop:variablecert}(iii) and Sion duality therefore give a maximizing
centre $c$ with
\[
 \min_\theta\bigl[H_a(\theta)-\Real(ce^{-i\theta})\bigr]
 \ge\frac{\sqrt{3}}{4}+0.0153.
\]
Since $R_{r_0}\ge H_a$, the same inequality holds for
$\varrho_{r_0}(c)$.
As in Section~\ref{sec:proof}, $R_{r_0}\le\operatorname{artanh}(r_0)$, so
$\abs{c}<0.211$. The pooled certificate gives
$\min_\theta R_{r_0}(\theta)\ge0.4149$, independently of the present case.
Thus Lemma~\ref{lem:disc}, with Bonk's theorem supplying univalence, applies
again. Taking the infimum over $F$ completes the proof.
\end{proof}

\section{Implementation and verification}\label{sec:impl}

The source code, certificate data, and reference-run records are archived at
\url{https://doi.org/10.5281/zenodo.21975862}. The computation has two stages,
deliberately separated.

\emph{Search} (\texttt{dual\_cert.py}, \texttt{make\_certificates.py}). Rows
\eqref{eq:row} are generated on a product grid of $56$ radii
$\rho \in [0.02, 0.95]$, clustered towards the outer edge, and $161$ angles
$\psi \in [0,\pi]$ (angles in $[\pi,2\pi]$ give identical rows). The
linearisation angles $\gamma_j$ are set to $\arg f^\star(\rho_je^{i\psi_j})$
for the current candidate optimiser $f^\star$, and refreshed over a handful of
rounds; $f^\star$ itself is recovered from the linear-program duals. Choosing
the multipliers so as to maximise the right-hand side of \eqref{eq:minorant} at
a prescribed anchor $(b_2,b_3)$, or to maximise the closed form
\eqref{eq:regionmin}, is itself a linear program, solved with
HiGHS through SciPy in double precision. Roughly $10^2$ of the several
thousand candidate rows receive a nonzero multiplier. None of this is trusted:
the output is a list of numbers $\lambda_j, \rho_j, \psi_j, \gamma_j$.
The search takes about five minutes for the coarse certificate and
ten minutes for the fine certificate on an Apple~M4.

\emph{Verification} (\texttt{certify\_bloch.py}). Working in Arb ball
arithmetic~\cite{Arb} at $40$ decimal digits, the program checks $\lambda_j \ge
0$, recomputes $d_k$ for $0 \le k \le K = 320$ directly from \eqref{eq:row},
bounds $B_k$ by rigorous enclosures of $((k+2)/k)^{k/2}(k+2)/2$, sums the tail
terms $T$ and $T_0$ in closed form, and evaluates \eqref{eq:e0},
Theorem~\ref{thm:a3cert} and \eqref{eq:regionmin}. Every quantity reported for
the pooled certificate is the outward-rounded lower end of the resulting
enclosure. The Arb computation underlying Proposition~\ref{prop:pooled} takes
less than a second. On the finer certificate, the key Arb enclosures are
\begin{align*}
 \text{upper bound for }\abs{a_3}&:\quad
   [3.28877762819\mathbin{\pm}6.87\cdot10^{-13}],\\
 \Phi_{r_0}(-1,0)-\frac{\sqrt{3}}{4}&\ge
   [0.0149407644273\mathbin{\pm}4.65\cdot10^{-14}],\\
 B-\frac{\sqrt{3}}{4}&\ge
   [0.0114402996202\mathbin{\pm}3.94\cdot10^{-15}],\\
 B&\ge[0.444453001512\mathbin{\pm}4.16\cdot10^{-13}],\\
 \min_\theta R_{r_0}(\theta)&\ge
   [0.41491575\mathbin{\pm}4.17\cdot10^{-9}].
\end{align*}
For the coarser certificate, the second enclosure is instead
$[0.0150346379669\mathbin{\pm}4.55\cdot10^{-14}]$, which is the value
quoted in Proposition~\ref{prop:phi10}. The last line is the bound
$\min_\theta R \ge 0.4149$ used both in Step 4 of the proof of
Proposition~\ref{prop:pooled} and for centre placement in the away branch of
Theorem~\ref{thm:main}.

\emph{Variable-radius verification}. The program and its stored data are
\begin{center}
\texttt{variable\_radius\_certificate.py},\qquad
\texttt{variable\_radius\_certificate.npz}.
\end{center}
The data contain the integral, point-value and additional away-branch
multipliers. The command
\begin{verbatim}
python variable_radius_certificate.py verify
\end{verbatim}
reconstructs the near certificates and the coefficient polyhedron in Arb. It
gives the outward-rounded bounds
\begin{align*}
 \text{uniform positivity margin}&\ge0.000936,\\
 \text{near moment slope margin}&\ge0.0331,\\
 \abs{c}&\le0.137<0.436\le\min_\theta R_{r_1}(\theta),\\
 \text{near moment gain}&\ge0.015304.
\end{align*}
The rigorous away calculation is split into independently restartable commands
\begin{verbatim}
python variable_radius_certificate.py rigorous --sector J --target 0.0153
\end{verbatim}
for $J=0,\ldots,23$. Each completed with output of the form
\begin{verbatim}
sector J: PASS > 0.0153; certified_boxes=..., open=0; elapsed=...
\end{verbatim}
Checkpoints store the uncertified box frontier, so interruption and resumption
do not alter the mathematical calculation. On completion the final zero
records that no box remains open.  A completed checkpoint contains only this
empty frontier and summary counters; it does not retain the discarded boxes
or their dual witnesses and is therefore not a replay certificate.
Independent verification requires starting without completed checkpoints and
rerunning all $24$ sectors.

The computational cost here is substantially greater than for
Proposition~\ref{prop:pooled}. The $24$ sector calculations are independent and
may be run in parallel. The calculations certified $11{,}443{,}518$ terminal
boxes in total. On the Apple M4 machine used for the final computation,
individual sectors took between roughly $2.5$ and $7$ hours, and the complete
verification of Theorem~\ref{thm:main} was finished in just under one day of
wall-clock time.

\emph{Implementation sanity check} (\texttt{validate\_certificate.py}).
As a check on conventions such as signs and row scaling, separate from the
rigorous verification above, the finished inequality \eqref{eq:minnum} was
evaluated on explicitly constructed members of $\Bcl$.

The constant in Proposition~\ref{prop:pooled} is not optimal even within the
single-minorant method. A coarser run certifies $0.0113729924$, against
$0.0114402996$ for the run reported here. Its limiting factor is the moment
relaxation of Step 2, which caps that compact argument near $0.0115$; the
multi-cut, coefficient-dependent-radius calculation is what passes this
ceiling.

\section{What is not proved}\label{sec:limits}

Theorem~\ref{thm:main} states the round target $0.0153$. The near branch has
certified gain
\[
 0.0153040536989472,
\]
leaving a margin of about $4.05\cdot10^{-6}$, and the Arb subdivision has
cleared all $24$ away sectors at the stated target.  We make no claim above
$0.0153$: the present near certificate alone would already have to be
strengthened to reach $0.01531$. Reconnaissance in
the away branch gives minima $0.015866072532$ on a $20\times20$ mesh and
$0.015648602353$ on a finer rotated-antipodal boundary search, so the near
branch, rather than the completed away subdivision, is now the immediate
limitation of this dichotomy.

\emph{Sharpness of $\Phi_{r_0}$.} For the truncated numerical program at
$(-1,0)$,
the certified dual gain $0.0150346$ is close to the computed primal gain
$0.015073$ from an independent solve. This comparison is reconnaissance only
and does not bound the exact value of $\Phi_{r_0}(-1,0)$ from above.


\emph{Higher coefficients at the fixed radius.} Pinning $a_4,a_5,\ldots$ is
not a promising way to improve the fixed-radius part of the argument. Indeed,
for every centre $c$ the shift terms cancel in two antipodal directions, and
hence
\[
 \min_\theta\bigl[R_{r_0}(\theta)-\Real(ce^{-i\theta})\bigr]
 \le \frac{R_{r_0}(0)+R_{r_0}(\pi)}{2}
 =\int_0^{r_0}g(t)\,dt,\qquad
 g(z)=\frac{f(z)+f(-z)}{2}.
\]
The function $g$ is an even member of $\Bcl$. Consequently every refinement
of the fixed-radius coefficient reduction, even one pinning all Taylor
coefficients, is bounded above by
\[
 \Phi_{\mathrm{even}}:=\inf\Bigl\{\int_0^{r_0}g(t)\,dt:
 g\in\Bcl,\ g\text{ even}\Bigr\}.
\]
Two-sided numerical calculations place
$\Phi_{\mathrm{even}}-\sqrt{3}/4$ between $0.0150660$ and $0.0151838$ on our
finest runs, already below the gain in Theorem~\ref{thm:main}. More directly,
adding $a_4$ to the pooled calculation changed the gain by only $3\cdot10^{-7}$,
and its optimiser set the new $a_4$ slope to zero; the published pooled
certificate already has its slopes from degree $4$ onward of order $10^{-8}$.
These two-sided and slope figures are reconnaissance. Independently, the
accompanying Arb program \texttt{certify\_fixed\_radius\_ceiling.py} verifies
that an explicit even polynomial $g$ of degree $44$, whose exact rational
coefficients are recorded in the Zenodo archive, belongs to $\Bcl$ and satisfies
\[
 \int_0^{r_0}g(t)\,dt-\frac{\sqrt{3}}{4}
 \le 0.0151972097090942<0.0152.
\]
Near the origin the Bloch constraint is reduced to positivity of a univariate
polynomial; the remainder of the disc is covered by adaptive Arb rectangles.
Thus the fixed-radius reduction is rigorously capped below $0.0152$, whereas
Theorem~\ref{thm:main} reaches $0.0153$.
There is no conflict with the away-branch calculation: this witness is even,
so $a_3=0$ and it belongs to the near branch, where the proof uses the larger
radius $r_1$. That larger-radius branch is precisely what evades the
fixed-radius obstruction, and suggests that radius information, rather than
higher coefficients, is the useful direction for further improvement.

\end{document}